\documentclass[11pt]{amsart}
\usepackage{epsfig}
\usepackage{graphics}
\usepackage{color}
\usepackage{dcpic, pictexwd}

\newtheorem{theorem}{Theorem}[section]
\newtheorem{Theorem}{Theorem}
\newtheorem{Corollary}[Theorem]{Corollary}
\newtheorem{lemma}[theorem]{Lemma}
\newtheorem{proposition}[theorem]{Proposition}
\newtheorem{corollary}[theorem]{Corollary}

\theoremstyle{definition}
\newtheorem{definition}[theorem]{Definition}

\newtheorem{remark}[theorem]{Remark}

\theoremstyle{remark}

 \def\Z{{\mathbb{Z}}}

 \def\C{{\mathbb{C}}}
 
\def\mod{{\rm Mod}}

 \def\GL{{\rm GL}}
 \def\Sp{{\rm Sp}}
 \def\Im{{\rm Im}}
 \def\PSp{{\rm PSp}}

\begin{document}

\newenvironment{prooff}{\medskip \par \noindent {\it Proof}\ }{\hfill
$\square$ \medskip \par}
    \def\sqr#1#2{{\vcenter{\hrule height.#2pt
        \hbox{\vrule width.#2pt height#1pt \kern#1pt
            \vrule width.#2pt}\hrule height.#2pt}}}
    \def\square{\mathchoice\sqr67\sqr67\sqr{2.1}6\sqr{1.5}6}
\def\pf#1{\medskip \par \noindent {\it #1.}\ }
\def\endpf{\hfill $\square$ \medskip \par}
\def\demo#1{\medskip \par \noindent {\it #1.}\ }
\def\enddemo{\medskip \par}
\def\qed{~\hfill$\square$}

 \title[Symplectic representation of mapping class group]
{The symplectic representation of the mapping class group is unique}

\author[Mustafa Korkmaz ]{Mustafa Korkmaz}

 \address{Department of Mathematics, Middle East Technical University,
 Ankara, Turkey\\ and Max-Planck-Institut f\"ur Mathematik, Bonn, Germany}
 \email{korkmaz@metu.edu.tr}

 \date{\today}

\begin{abstract}
Any nontrivial homomorphism from the mapping class group of an orientable surface of genus $g\geq 3$ to
$\GL(2g,\C)$ is conjugate to the standard symplectic representation. It is also shown that
the mapping class group has no faithful linear representation in dimensions less than or equal to $3g-3$.
\end{abstract}

 \maketitle

  \setcounter{secnumdepth}{2}
 \setcounter{section}{0}
\section{Introduction and the results}

For a compact connected orientable surface $S$ of genus $g$ with a finite (possibly empty) set of marked points
in the interior, we define the \emph{mapping class group} $\mod (S)$ to be the group of isotopy classes of orientation--preserving
self--diffeomorphisms of $S$. Diffeomorphisms and isotopies are assumed to be the identity on the boundary and on the marked points.
Gluing a disk along each boundary component of $S$ and forgetting the marked points give a
closed surface $\bar S$ and a natural surjective map $\mod(S)\to\mod(\bar S)$ between the mapping class groups.

After fixing a basis for the first (integral) homology $H_1(\bar S,\Z)$ of $\bar S$, the action
of $\mod (\bar S)$ on $H_1(\bar S,\Z)$ gives rise to a
surjective homomorphism $\mod (\bar S) \to \Sp (2g,\Z)$. Precomposing this map with $\mod(S)\to\mod(\bar S)$ and
postcomposing with the inclusion $\Sp (2g,\Z)\hookrightarrow \GL (2g,\C)$
give a map $P:\mod (S) \to \GL (2g,\C)$. Our first result shows that
the map $P$ is the only nontrivial representation of $\mod (S)$ in this dimension:

\begin{Theorem} \label{thm:main1}
Let $g \geq 3$ and let $\phi : \mod (S)\to \GL (2g,\C)$ be a group homomorphism.
Then $\phi$ is either trivial or conjugate to the homomorphism $P: \mod (S)\to \GL (2g,\C)$.
\end{Theorem}

We say that two homomorphisms $\psi_1$ and $\psi_2$ from a group $G$ to
a group $H$ are called \emph{conjugate} if there exists an element $h\in H$ such that $\psi_2(x)=h\psi_1(x) h^{-1}$ for all $x\in G$.

One of the outstanding unsolved problems in the theory of mapping class groups is the existence of a faithful
representation $\mod (S)\to \GL (n,\C)$. Our second result shows that in dimensions $n\leq 3g-3$,
there is no faithful linear representation of the mapping class group.

\begin{Theorem} \label{thm:main2}
Let $g \geq 3$ and let $n\leq 3g-3$. Then there is no injective homomorphism
$\mod (S)\to \GL (n,\C)$.
\end{Theorem}

The linearity problem for the braid group, which is the mapping class group of
a disk with marked points in the context of this paper, was solved by Bigelow~\cite{bigelow} and also by
Krammer~\cite{krammer1,krammer2}. In~\cite{k00tjm}, using the linearity of braid groups,
the author proved that the mapping class group of
the closed surface of genus $2$ and of sphere with marked points are linear.
These results were also obtained by Bigelow--Budney~\cite{bb}.

The first result on the nonexistence of faithful linear representations of the mapping class group
was obtained by Farb--Lubotzky--Minsky~\cite{flm} who
proved that no homomorphism from a subgroup of finite index of the mapping class group
into $\GL(n,\C)$ is injective for $n<2\sqrt{g-1}$. This was improved first by
Funar~\cite{funar} who showed that every map from the mapping class group
into ${\rm SL} (n,\C)$ has finite image for $n\leq \sqrt{g+1}$.
This was improved further by Franks--Handel in~\cite{fh} and by the author in~\cite{k},
which can be rephrased as follows:

\begin{Theorem} \cite{fh,k} \label{thm:2g-1}
  Let $g\geq 2$, $n\leq 2g-1$, and let $\phi:\mod (S)\to \GL (n,\C)$ be a homomorphism. Then $\phi$
  factors through the natural quotient map $\mod (S)\to H_1 (\mod (S),\Z)$. In particular, ${\rm Im}(\phi)$  is
  trivial if $g\geq 3$, and is a quotient of the cyclic group $\Z_{10}$ of order $10$ if $g=2$.
\end{Theorem}

We state a few algebraic corollaries to above theorems, some of which might be known to the experts.

\begin{Corollary}
If $g\geq 3$ then up to the conjugation by an element of $\GL (2g,\C)$ there are exactly two homomorphisms
$\Sp (2g,\Z)\to \GL (2g,\C)$, the trivial homomorphism and the injection given by the inclusion.
\end{Corollary}

\begin{Corollary}
If $g\geq 3$ and $n\leq 2g$ then every homomorphism from $\PSp (2g,\Z)$ to $\GL (n,\C)$ is trivial.
\end{Corollary}

\begin{Corollary} If $g\geq 3$ and if $Q$ is a finite quotient of $\mod(S)$,
then every homomorphism $Q\to \GL (2g,\C)$ is trivial.
\end{Corollary}

In order to prove these corollaries, consider the composition the given homomorphism $H \to \GL (2g,\C)$ with
the natural surjective map $\mod(S)\to H$ and apply Theorem~\ref{thm:main1} or Theorem~\ref{thm:2g-1}.

Here is an outline of the paper. In Section~\ref{sec:prlm}, we give some lemmas on matrices.
Section~\ref{sec:mcg} gives the relevant background from the theory of mapping class groups and surface topology.
Section~\ref{sec:eigvalues} investigates the properties of eigenvalues and eigenspaces of the image of a
Dehn twist about a nonseparating simple closed curve.
The main result in this section is Lemma~\ref{lem:L=AL=B}, which constitutes a major step in the
proofs of both theorems. The first main result, Theorem~\ref{thm:main1}, is proved in
Section~\ref{sec:symprep}. We give the proof for $g\geq 4$ first, and an then outline for $g=3$.
Finally,  Theorem~\ref{thm:main2} is proved in Section~\ref{sec:nonfaithful}.

\medskip

{\bf Acknowledgments.}
This paper was written while I was visiting Max-Planck Institut f\"ur Mathematik in Bonn
on leave from Middle East Technical University.
I thank MPIM for its generous support and wonderful research environment. I also thank
\c{C}a\u{g}r{\i} Karakurt for the discussions and for his comments.

\bigskip


\section{Preliminaries in linear algebra} \label{sec:prlm}
In this section, we give the necessary definitions, setup the notation and prove some useful lemmas on matrices.
We denote by $\bar S$ the closed surface obtained from $S$ by gluing a disk along each boundary component.

Throughout the paper, the letters $U$ and $\widehat U$ will always denote the $2\times 2$ matrices
$U = \left(
           \begin{array}{cc}
             1 & 1 \\
             0 & 1 \\
           \end{array}
         \right)$ and
$\widehat U = \left(
           \begin{array}{cc}
             1 & 0 \\
             -1 & 1 \\
           \end{array}
         \right)$.

\begin{definition} \label{def:A-iB-i}
For each $i=1,2,\ldots, g$, we define the matrices $A_i$ to be the $2g\times 2g$ block diagonal matrix
Diag$(I_2,\ldots,I_2,U,I_2,\ldots,I_2)$ and
$B_i$ to be the $2g\times 2g$ block diagonal matrix Diag$(I_2,\ldots,I_2,\widehat U,I_2,\ldots,I_2)$.
Here, $U$ and $\widehat U$ are in the $i^{\rm th}$ block on the diagonal, and $I_2$ is the $2\times 2$ identity matrix.
\end{definition}

We note that $A_i$ (resp. $B_i$) is the symplectic matrix of the action of the Dehn twist $t_{a_i}$ (resp. $t_{b_i}$)
on the first homology group of $\bar S$ with respect to the basis $\{a_1,b_1,a_2,b_2,\ldots,a_g,b_g\}$
of $H_1(\bar S,\Z)$,
where $a_i$ and $b_i$ are the (oriented) simple closed curves given in Figure~\ref{abcurves}.
Recall that when we consider a Dehn twist about a curve, the orientation of the curve
is unimportant.

\bigskip

\begin{lemma} \label{lem:2XYZ}
Let $X$, $Y$ and $Z$ be, respectively, $2\times k$, $k\times 2$ and $2\times 2$, matrices with entries in $\C$.
    \begin{enumerate}
        \item [(i)] If \ $UX=\widehat U X = X$, then $X=0$.
        \item [(ii)]  If \ $YU=Y\widehat U =  Y$, then $Y=0$.
        \item [(iii)] If \ $ZU=UZ$ and \ $Z \widehat U = \widehat U Z$, then $Z=aI$.
    \end{enumerate}
\end{lemma}
\begin{proof}
The proof is straight forward.
\end{proof}

\medskip

\begin{lemma}
\label{lem:nXYZ}
Let $X$, $Y$ and $Z$ be matrices with entries in $\C$ such that the given multiplication are defined:
    \begin{enumerate}
        \item [(i)] If \ $A_iX=B_iX=X$ for all $i$, then $X=0$.
        \item [(ii)] If \ $YA_i=YB_i=Y$ for all $i$, then $Y=0$.
        \item [(iii)] If \ $ZA_i=A_iZ$ and \ $Z B_i = B_i Z$ for all  $i$, then $Z$ is equal to a diagonal matrix
        ${\rm Diag}(a_1I_2,a_2I_2,\ldots,a_gI_2)$ for some $a_i\in\C$.
    \end{enumerate}
\end{lemma}
\begin{proof}
This lemma is a slight generalization of Lemma~\ref{lem:2XYZ}, and may be proved easily by induction
on $g$.
\end{proof}

\medskip

\begin{lemma} \label{lem:X=U}
 Let $X\in \GL (2,\C)$. Suppose that
 $XU=UX$ and   $X\widehat U X=\widehat U X\widehat U $. Then $X=U$.
 \end{lemma}
\begin{proof}
Let $X = \left(
           \begin{array}{cc}
             a & b \\
             c & d \\
           \end{array}
         \right)$.
The equality $XU=UX$ implies that $c=0$ and $d=a$, so that
$X = \left(
           \begin{array}{cc}
             a & b \\
             0 & a \\
           \end{array}
         \right)$.
Note that $a\neq 0$.
Now the equality $X\widehat U X=\widehat U X\widehat U $ gives the equations
\begin{eqnarray*}
a(a-b)&=&a-b,\\
a^2&=&2a-b, \\
b(2a-b)&=& b.
\end{eqnarray*}
The only solution of these equations is $a=b=1$.
\end{proof}

\begin{lemma}  \label{lem:X=A-1}
Let $X\in \GL (2g,\C)$. Suppose that
\begin{enumerate}
  \item [(i)] all eigenvalues of $X$ are equal to $1$,
  \item [(ii)] $XA_i=A_iX$ for all $i=1,2,\ldots, g$,
  \item [(iii)] $XB_j=B_jX$ for all $j=2,3,\ldots, g$, and
  \item [(iv)] $XB_{1}X=B_{1}XB_{1}$.
\end{enumerate}
Then $X=A_1$.
\end{lemma}
\begin{proof} If $g=1$ then the lemma reduces to Lemma~\ref{lem:X=U}. So suppose that
$g\geq 2$.

Let us write $X = \left( \begin{array}{cc}
               X_1 & X_2 \\ X_3 & X_4 \\
        \end{array} \right)$, where $X_1$ and $X_4$ are, respectively, $2\times 2$ and $(2g-2)\times (2g-2)$
matrices. For each $i\geq 2$, we define $\bar A_i$ and $\bar B_i$ by
$A_i = \left( \begin{array}{cc}
              I_2 & 0 \\ 0 & \bar A_i \\
           \end{array} \right)$,
$B_i = \left( \begin{array}{cc}
             I_2 & 0 \\ 0 & \bar B_i \\
           \end{array} \right)$.

It is easy to see that, for $i\geq 2$, the equations $XA_i=A_iX$ and $XB_i=B_iX$ imply that
\begin{itemize}
  \item  $X_2\bar A_i =X_2\bar B_i =X_2$,
  \item  $\bar A_iX_3 =\bar B_i X_3=X_3$, and
  \item  $X_4 \bar A_i=\bar A_i X_4$ and $X_4 \bar B_i=\bar B_i X_4$.
\end{itemize}
It now follows from Lemma~\ref{lem:nXYZ} and the assumption that all eigenvalues of $X$ are equal to $1$ that
$X_2=0,X_3=0$ and $X_4=I_{2g-2}$.

Moreover, from the equations $XA_1=A_1X$ and $XB_1X=B_1XB_1$, we obtain
\begin{itemize}
  \item $X_1U=UX_1$, and
  \item $X_1\widehat UX_1= \widehat U X_1 \widehat U$.
\end{itemize}
Now these two equations and Lemma~\ref{lem:X=U} give us the equality
$X_1=U$, so that $X=A_1$, which is the desired result.
\end{proof}

\begin{remark} \label{rem:X=A-i}
The above proof may be modified easily to prove the following version of
Lemma~\ref{lem:X=A-1}: Suppose that
\begin{enumerate}
  \item [(i)] all eigenvalues of $X\in \GL (2g,\C)$ are equal to $1$,
  \item [(ii)] $XA_i=A_iX$ for all $i=1,2,\ldots, g$,
  \item [(iii)] $XB_j=B_jX$ for all $1\leq j \leq g$ with $j\neq k$, and
  \item [(iv)] $XB_{k}X=B_{k}XB_{k}$.
\end{enumerate}
Then $X=A_k$. The lemma holds true when the roles of $A_i$ and $B_i$ are exchanged as well.
\end{remark}

We state the following facts which will be used throughout the paper.
\begin{lemma}
Let $r$ and $s$ be two positive integers. Then the subgroup of $\GL (r+s,\C)$
consisting of the matrices of the from $\left(  \begin{array}{cc}
             I_r & * \\
             0 & I_s \\
           \end{array} \right)$
is abelian.
\end{lemma}

\begin{lemma}
The subgroup of $\GL (m,\C)$ consisting of upper triangular matrices
is solvable.
\end{lemma}

\bigskip


\section{The mapping class group results}
\label{sec:mcg}
Let $S$ be a compact connected oriented surface of genus $g$ with
a finite number of marked points in the interior.
In this section we state the results from the theory of mapping class groups that are needed in
the proofs of our main results. For further information on mapping class groups,
the reader is referred to~\cite{iv}, or~\cite{fm}. For a simple closed curve $a$ on $S$
we denote by $t_a$ the (isotopy class of the) right Dehn twist about $a$.

\begin{theorem}
\label{thm:dualeqv} $($\cite{km}, {\rm Theorem} $1.2)$
 Let $g\geq 1$ and let $a$ and
$b$ be two nonseparating simple closed curves on $S$. Then there is a sequence
\[
a=a_0,a_1,a_2,\ldots,a_k=b
\]
of nonseparating simple closed curves such that $a_{i-1}$ intersects $a_i$ at only one point
\end{theorem}

It is known that for $g\geq 2$, the group $\mod (S)$ is generated Dehn twists about nonseparating
simple closed curves. If the surface is closed this is due to Dehn~\cite{dehn} and Lickorish~\cite{lick}.
We record this fact and three well--known relations among Dehn twists.

\begin{theorem}
  If  $g\geq 2$ then the mapping class group $\mod (S)$ is generated Dehn twists about nonseparating
  simple closed curves on $S$.
\end{theorem}

We note that the above theorem does not hold true for $g=1$. More precisely, if $S$ is a torus
with at least two boundary components, then Dehn twists about nonseparating simple closed curves
are not sufficient to generate $\mod(S)$: one also needs Dehn twists about the curves parallel to
the boundary components~\cite{gervais, k02tjm}.

\begin{lemma}
  Let $a$ and $b$ be two simple closed curves on $S$, and let $t_a$ and $t_b$
  denote the right Dehn twists about them.
  \begin{enumerate}
    \item If $a$ and $b$ are disjoint, then $t_a$ and $t_b$ commute.
    \item (\textbf{Braid relation}) If $a$ intersects $b$ transversely at one point, then they satisfy the
    braid relation $t_at_bt_a=t_bt_at_b$.
    \item (\textbf{Lantern relation})
    Consider a sphere $X$ with four boundary components $a,b,c,d$. Let $x,y,z$ be three simple closed curves on
    $X$ as shown in Figure~\ref{fig:lantern}. Then the Dehn twists about them satisfy the lantern relation
    \[t_a t_b t_c t_d = t_x t_y t_z.\]
  \end{enumerate}
\end{lemma}

\begin{figure}[hbt]
 \begin{center}
      \includegraphics[width=4cm]{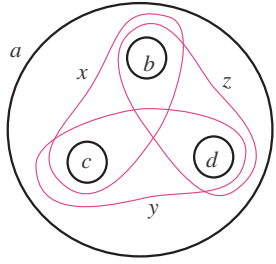}
      \caption{The curves of the lantern relation.}
      \label{fig:lantern}
   \end{center}
 \end{figure}

We remind that the first homology group $H_1(G;\Z)$ of a group $G$ is isomorphic to
the abelianization $G/[G,G]$, where $[G,G]$ is the (normal) subgroup of
$G$ generated by all commutators $xyx^{-1}y^{-1}$, where $x,y\in G$.

\begin{theorem}  \label{thm:H-1} $($\cite{k02tjm}, {\rm Theorem} $5.1)$
  If $g\geq 2$ then first homology group
  $H_1(\mod (S);\Z)$ is
  \begin{itemize}
    \item[(1)] trivial if $g\geq 3$, and
    \item[(2)]  isomorphic to the cyclic group of order $10$ if $g=2$.
  \end{itemize}
\end{theorem}

If $g\geq 3$ then the group $\mod(S)$ is perfect. In the case $g=2$, $\mod(S)$ is not perfect
but its commutator subgroup is perfect.

\begin{theorem}  \label{thm:G'prf} $($\cite{km}, {\rm Theorem} $4.2)$
  If $g\geq 2$ then the commutator subgroup of $\mod (S)$ is perfect.
\end{theorem}

\begin{theorem}  \label{thm:G'} $($\cite{km}, {\rm Theorem} $2.7)$
Let  $g\geq 2$ and let $a$ and $b$ be two nonseparating simple closed curves on $S$ intersecting
at one point. Then the commutator subgroup of $\mod (S)$ is generated normally by $t_at_b^{-1}$.
\end{theorem}

Note that since any two Dehn twists about nonseparating simple closed curves are conjugate in $\mod (S)$,
they represent the same element in the group $H_1(\mod (S);\Z)$. In particular, we conclude the next lemma.

\begin{lemma}
\label{lem:abelianrep}
Let $g\geq 1$, and let $b$ and $c$ be two nonseparating simple closed curves on $S$.
If $H$ is an abelian group and if $\phi :\mod (S)\to H$ is a homomorphism, then $\phi (t_b)= \phi (t_c)$.
\end{lemma}

Let us write $S=S_{g,r}^p$, for the moment, for the surface of genus $g\geq 2$ with $p\geq 0$ boundary
and with $r\geq 0$ marked points in the interior. For $r\geq 1$ by forgetting one of the marked points,
we get a short exact sequence,
\begin{equation}\label{eqn:exsqn1}
  1\longrightarrow \pi_1 (S_{g,r-1}^{p})  \longrightarrow   \mod (S_{g,r}^{p}) \longrightarrow
   \mod (S_{g,r-1}^{p})\longrightarrow 1,
\end{equation}
where the map from the fundamental group is obtained by pushing the base point along the given path.

For $p\geq 1$, by blowing down a given boundary component $d$ to a marked point,
we get a short exact sequence
\begin{equation}\label{eqn:exsqn1}
  1\longrightarrow \Z  \longrightarrow   \mod (S_{g,r}^{p}) \longrightarrow
   \mod (S_{g,r+1}^{p-1})\longrightarrow 1,
\end{equation}
where the group $\Z$ is generated by the Dehn twist $t_d$ about the boundary component $d$.
These two exact sequences are called \emph{Birman's exact sequences}.

\begin{figure}[hbt]
 \begin{center}
      \includegraphics[width=10cm]{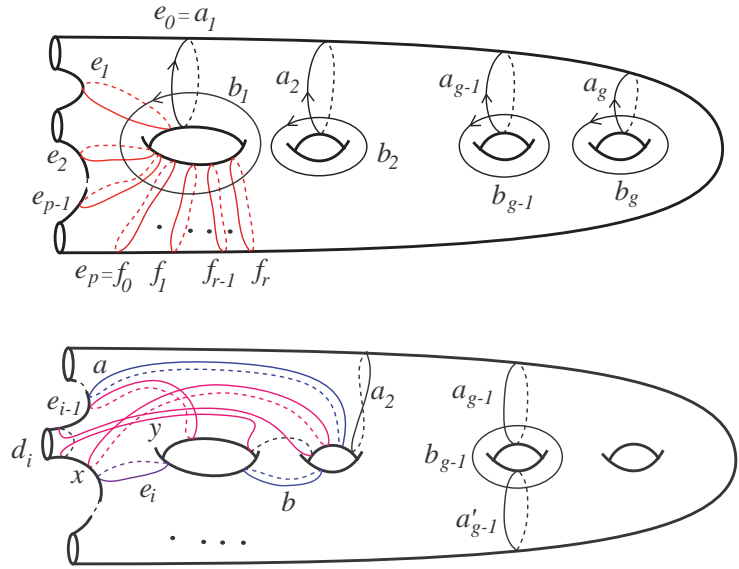}
      \caption{Various curves on the surface and a lantern.}
      \label{abcurves}
   \end{center}
 \end{figure}

Let $g\geq 2$, and let $\bar S=S_{g,0}^0$ be the closed surface obtained from $S$ by gluing a disk along
each boundary component. Let $\pi:\mod (S)\to\mod (\bar S)$ be the natural
surjective homomorphism obtained by extending
diffeomorphism $S\to S$ to $\bar S\to \bar S$ by the identity on the disks glued.
Let $d_1,d_2,\ldots,d_p$ be (the simple closed curves parallel to) the boundary components of $S$. Consider the
simple closed curves $e_0,e_1,\ldots, e_p$, $f_0,f_1,\ldots, f_r$ and $a'_{g-1}$ illustrated in
Figure~\ref{abcurves}. Hence, the union $d_i \cup e_{i-1} \cup e_i$ bounds a pair of pants, $f_{j-1}\cup f_j$
bounds an annulus with a marked point, and  $a_{g-1}\cup a'_{g-1}$ bounds a subsurface diffeomorphic to a torus
with two boundary components.
Clearly, for $i=1,2,\ldots,p$ and $j=1,2,\ldots,r$, the mapping classes $t_{d_i}, t_{e_{i-1}}t_{e_i}^{-1}$
and $t_{f_{j-1}}t_{f_j}^{-1}$ are contained in the kernel of $\pi$.
It can be shown by using Birman's exact sequences that the kernel of $\pi$ is generated normally
by
\begin{equation*}
\{ t_{d_i}, t_{e_{i-1}}t_{e_i}^{-1}, t_{f_{j-1}}t_{f_j}^{-1} : 1\leq i\leq p, 1\leq j\leq r\}.
\end{equation*}
We state this fact as a proposition which will be useful for us. We also show that we may omit the Dehn twists
about the boundary parallel curves.

\begin{proposition} \label{prop:kernel}
The kernel of the natural surjective homomorphism $\pi:\mod (S)\to\mod (\bar S)$ is generated normally by
the set
\begin{equation}\label{eqn:ngenset}
\{ t_{e_{i-1}}t_{e_i}^{-1}, t_{f_{j-1}}t_{f_j}^{-1} : 1\leq i\leq p, 1\leq j\leq r\}.
\end{equation}
\end{proposition}
\begin{proof}
For the proof, it suffices to prove that each $t_{d_i}$ can be written as a product
of conjugates of elements in~\eqref{eqn:ngenset}. We use the lantern relation for this.
Consider the curves in the bottom picture of Figure~\ref{abcurves}. The curves $d_i, e_i, a,b$ bound a sphere
with four boundary components. The curves on this sphere satisfy the lantern relation
\begin{equation*}
    t_{d_i} t_{e_i} t_{a} t_{b} = t_{e_{i-1}} t_{x} t_{y},
\end{equation*}
which may be rewritten as
\begin{equation*}
    t_{d_i}  = (t_{e_{i-1}}t_{e_i}^{-1})  (t_{x}t_{a}^{-1}) (t_{y} t_{b}^{-1} ).
\end{equation*}
Since the pair $(x,a)$ can be mapped $(e_{i-1},e_i)$ by a diffeomorphism of $S$,
it follows that the element $t_{x}t_{a}^{-1}$ is conjugate to $t_{e_{i-1}}t_{e_i}^{-1}$.
By the similar reasoning, $t_{y} t_{b}^{-1}$ is also conjugate to $t_{e_{i-1}}t_{e_i}^{-1}$.
\end{proof}

\bigskip


\section{Eigenvalues and eigenspaces of $\phi(t_a)$}
\label{sec:eigvalues}

Let $\phi:\mod(S) \to \GL(m,\C)$ be a homomorphism. For a simple closed curve $a$ on $S$,
we denote by $L_a$ the image $\phi (t_a)$ of the Dehn twist $t_a$. If $\lambda$ is an eigenvalue of
$L_a$, the eigenspace corresponding to $\lambda$ is denoted by $E^a_\lambda$.

For each $i=1,2,\ldots , g$, let $a_i$ and $b_i$ be the nonseparating simple closed curves on $S$
shown is Figure~\ref{abcurves}. We may consider them on $\bar S$ as well, where $\bar S$ is
the closed surface obtained from $S$ by gluing a disk along each boundary component.
The homology classes of the orient curves $a_i$ and $b_i$ form a basis for the first homology group
$H_1(\bar S,\Z)$ of $\bar S$.  In order to avoid double subscript,
we write $L_{2i-1}=\phi (t_{a_i})$ and $L_{2i}=\phi (t_{b_i})$.

For an eigenvalue $\lambda$ of a matrix $M$, let $\lambda_\# (M)$ denote the multiplicity of $\lambda$
in the characteristic polynomial of $M$. We will omit $M$ from the notation and write $\lambda_\#$ only;
the matrix $M$ will always be clear from the context.

\medskip

\begin{lemma}
Let $L$ and $M$ be two linear automorphisms of $\C^m$ with $LM=ML$. If $\lambda$ is an eigenvalue of
$L$, then $\ker (L-\lambda I)^k$ is $M$--invariant for all $k\geq 1$. In particular, $E_\lambda= \ker (L-\lambda I)$
is $M$--invariant.
\end{lemma}

\medskip

\begin{lemma} \label{lem:eigsp=}~$($\cite{k}$)$
Let $\phi:\mod (S)\to \GL (m,\C)$ be a homomorphism.
Let $a,b,c,d$ be nonseparating simple closed curves on $S$ such that there exists
an orientation--preserving diffeomorphism $f:S\to S$ with $f(c)=a$ and $f(d)=b$.
Suppose that $\lambda$ is an eigenvalue of $L_a=\phi (t_a)$.
Then $E_\lambda^a=E_\lambda^b$ if and only if $E_\lambda^c=E_\lambda^d$.
\end{lemma}

\medskip

\begin{lemma} \label{lem:E-a=E-b}
Let $S$ be a compact connected oriented surface of genus $g\geq 1$ and let $\phi:\mod (S)\to \GL (m,\C)$
be a homomorphism. Suppose that there are two nonseparating simple closed curves $a$ and $b$ on $S$ intersecting
at one point such that $E^a_\lambda=E^b_\lambda$ for some eigenvalue $\lambda$ of $L_a$ and $L_b$.
Then $E^a_\lambda$ is $\mod (S)$--invariant. That is, $\phi (f) (E^a_\lambda)=E^a_\lambda$ for all $f\in \mod (S)$.
\end{lemma}
\begin{proof}
Let $x$ be a nonseparating simple closed curve on $S$. By Theorem~\ref{thm:dualeqv},
there is a sequence $a=a_0,a_1,a_2,\ldots,a_k=x$ of nonseparating simple closed curves such that
$a_{i-1}$ intersects $a_i$ at one point for all $1\leq i\leq k$. Since, by the classification of surfaces,
there exists a diffeomorphism $f_i$ of $S$ mapping $(a,b)$ to $(a_{i-1},a_i)$, we have
$E_\lambda^{a_{i-1}} = E_\lambda^{a_{i}}$ by Lemma~\ref{lem:eigsp=}. It follows that
$E_\lambda^x = E_\lambda^a$ for all nonseparating simple closed curves $x$.
Since $\mod (S)$ is generated by nonseparating Dehn twists, we conclude that the subspace $E_\lambda^a $ is
$\mod(S)$--invariant.
\end{proof}

\medskip

\begin{lemma} \label{lem:ev=1for2g-1}
Let $g\geq 3$, let $\phi:\mod(S) \to \GL(m,\C)$ be a homomorphism and let
$a$ be a nonseparating simple closed curve on $S$. Suppose that $L_a$ has only one eigenvalue
$\lambda$ with $\dim (E^{a}_\lambda)=m-1$. Then $\lambda=1$.
\end{lemma}
\begin{proof}
Choose six nonseparating simple closed curves $c_1,c_2,c_3,c_4,c_5,c_6$ on $S$ disjoint from $a$ such that
we have the following lantern relation:
\[ t_at_{c_1}t_{c_2}t_{c_3}=t_{c_4}t_{c_5}t_{c_6}.
\]

Let $\alpha=\{v_1,v_2,\ldots, v_m\}$ be a basis of $\C^m$ with $v_j\in E^a_\lambda$ for $j\geq 2$,
so that \[ L_a= \left(
  \begin{array}{cc}
     \lambda & 0 \\
     \star  & \lambda I_{m-1} \\
  \end{array}
\right)
\]
with respect to $\alpha$.
Since each Dehn twist $t_{c_i}$ is conjugate to $t_a$, $L_{c_i}=\phi (t_{c_i})$ is conjugate to $L_a$
and, hence, it has unique eigenvalue $\lambda$.
Since $L_{c_i}$ commutes with $L_{a}$, it preserves the eigenspace $E^a_\lambda$; $L_{c_i}(E^a_\lambda)=E^a_\lambda$.
Thus, the matrix of $L_{c_i}$ with respect to the basis $\alpha$ is
\[ L_{c_i}=  \left(
  \begin{array}{cc}
     \lambda & 0 \\
     \star  & C_i \\
  \end{array}
\right),
\]
where $C_i$ is a square matrix of size $m-1$. Now the lantern relation
\[ L_aL_{c_1}L_{c_2}L_{c_3}=L_{c_4}L_{c_5}L_{c_6}
\]
implies that $\lambda^4=\lambda^3$, giving $\lambda=1$.
\end{proof}

\bigskip

\begin{lemma} \label{lem:ev=1}
Let $g\geq 3$, let $\phi :\mod (S)\to \GL (m,\C)$ be a homomorphism and let $a$ be
a nonseparating simple closed curve on $S$. If $\mu$ is an eigenvalue of $L_a=\phi (t_a)$ with
$\mu_\#  \leq 2g-3$ then $\mu=1$ and the dimension of the eigenspace $E^a_\mu$ is $\mu_\#$.
\end{lemma}
\begin{proof}
If $m\leq 2g-1$ then the image of $\phi$ is trivial by Theorem~\ref{thm:2g-1}. In particular,
all eigenvalues of $L_a$ are equal to $1$. Now assume that $m \geq  2g$.

Let $\mu=\mu_0,\mu_1,\mu_2,\ldots,\mu_s$ be all distinct eigenvalues of $L_a$.
Set
\[ K=\ker (L_a-\mu_0 I)^m,\mbox{ and }K'=\bigoplus_{i=1}^s \ker(L_a-\mu_i I)^m\]
so that $\dim (K) =\mu_\#$, and that $\C^m = K\oplus K'$. Let $\alpha$ be a basis $K$ and $\alpha'$ be a basis of $K'$
such that with respect to the basis $\alpha\cup\alpha'$ of $\C^m$, the matrix of $L_a$
\[ L_a= \left(
  \begin{array}{cc}
     A & 0 \\
    0  & A' \\
  \end{array}
\right),
\]
is in Jordan form, where $A$ is a square matrix of size $\mu_\#$.

Let $R$ be the complement of a regular neighborhood of $a$. Extending self--diffeomorphisms
of $R$ to $S$ by the identity gives a homomorphism $q:\mod (R)\to \mod(S)$. Let $\psi=\phi\, q$
and let $a'$ be a simple closed curve on $R$ which is isotopic to $a$ on $S$.
If $f\in \mod (R)$ is any element, $\psi (f)$
commutes with $L_a$ so that it preserves the subspaces $K$ and $K'$. Hence, the matrix of $\psi (f) $ is
\[ \psi (f)= \left(
  \begin{array}{cc}
     F & 0 \\
    0  & F' \\
  \end{array}
\right),
\]
where $F$ is a square matrix of size $\mu_\#$.

Now, the correspondence $f\mapsto F$ defines a homomorphism $\bar \psi: \mod (R) \to \GL (K)=\GL (\mu_\#,\C)$. Since
$\mu_\# \leq 2g-3=2(g-1)-1$ and the genus of $R$ is $g-1\geq 2$,
the image of the map $\bar \psi$ is cyclic by Theorem~\ref{thm:2g-1}.
It is easy to find six simple closed curves $b,c,d,x,y,z$ which are nonseparating on $R$
such that there is the lantern relation $t_{a'}t_bt_ct_d=t_xt_yt_z$, and that
each of $t_xt_b^{-1}$, $t_yt_c^{-1}$ and $t_zt_d^{-1}$ is a commutator in $\mod(R)$.
Since $t_{a'}=t_xt_b^{-1}t_yt_c^{-1}t_zt_d^{-1}$, it follows that $t_{a'}$
is contained in the commutator subgroup of $\mod (R)$. In particular, $\bar \psi (t_{a'})=I$. As a result of this we have
\[ L_a=L_{a'}= \left( \begin{array}{cc}
     \bar \psi (t_{a'}) & 0 \\
    0  & A' \\  \end{array} \right)
    = \left( \begin{array}{cc}
     I & 0 \\
    0  & A' \\  \end{array} \right).\]
Hence, $\mu=1$ and $\dim (E^a_\mu)=\mu_\#$.
\end{proof}

\bigskip

\begin{corollary} \label{cor:2ev}
Let $g,\phi$ and $a$ be as in Lemma~\ref{lem:ev=1}. If $m\leq 4g-5$ then $L_a$
has at most two eigenvalues.
\end{corollary}

\bigskip
\subsection{The main lemma}
The main step in the proof of Theorem~\ref{thm:main1} is the next lemma, which is also used in the proof of
Theorem~\ref{thm:main2}.

\begin{lemma} \label{lem:L=AL=B}
 Let $g\geq 1$, $m\geq 2g$ and let $\phi:\mod (S)\to \GL (m,\C)$ be a homomorphism.
 Let $a$ be a nonseparating simple closed curve on $S$. Suppose that the Jordan form of $L_a$ is
 \begin{equation}\label{eqn:Jordan}
     \left(
    \begin{array}{c|c}
        \begin{array}{cc}
          1 & 1 \\
          0 & 1 \\
        \end{array}
       & 0 \\ \hline
      0 & I_{m-2} \\
    \end{array}
  \right).
 \end{equation}
Suppose also that there exists a nonseparating simple closed curve $b$ intersecting $a$ at one point
 such that $E^a_1\neq E^b_1$. Then there is a basis of $\C^m$ with respect to which
 \[ L_{a_i}=\left(
  \begin{array}{cc}
    A_i & 0 \\
    0 & I \\
  \end{array}
\right) \mbox{ and } L_{b_i}=\left(
  \begin{array}{cc}
    B_i & 0 \\
    0 & I \\
  \end{array}
\right),\] where $I$ is the identity matrix of size $m-2g$. (See Definition~\ref{def:A-iB-i} for $A_i$ and $B_i$, and Figure~\ref{abcurves}
for the curves $a_i$ and $b_i$.)
 \end{lemma}

\begin{proof}
For $i=1,2,\ldots,g$, we set
  \begin{center} $L_{2i-1}=L_{a_i}$, $L_{2i}=L_{b_i}$,
  $E^{2i-1}=E^{a_i}_1$ and $E^{2i}=E^{b_i}_1$. \end{center}
We also set \[ \widetilde A_i=\left(
  \begin{array}{cc}
    A_i & 0 \\
    0 & I \\
  \end{array}
\right) \mbox{ and }\widetilde B_i=\left(
  \begin{array}{cc}
    B_i & 0 \\
    0 & I \\
  \end{array}
\right).\]
Since $E^a_1\neq E^b_1$ and since there exists a homeomorphism mapping $(a,b)$ to $(a_i,b_i)$,
 by Lemma~\ref{lem:eigsp=} we have $E^{2i-1}\neq E^{2i} $ for all $1\leq i\leq g$.
Note that $\dim (E^j)=m-1$ for all $1\leq j\leq 2g$.

Since $E^1\neq E^2$,  the dimension of $W_1=E^{1}\cap E^{2}$ is $m-2$. Let $\{v_3,v_4,\ldots,v_m\}$ be a
basis of $W_1$. Choose $v_1\in E^1\setminus E^2$ and $v_2\in E^2\setminus E^1$, so that
 $\beta_0=\{v_1,v_2,\ldots,v_m\}$ is an ordered basis of $\C^m$.
 With respect to $\beta_0$, we have
\[L_1=
\left(
  \begin{array}{c|c}
       \begin{array}{cc}
        1 & x_1 \\
        0 & 1 \\
       \end{array}
     & 0 \\ \hline
      \begin{array}{cc}
        0 & X  \\
      \end{array}
     & I_{m-2} \\
  \end{array}
\right), \ \
L_2=
\left(
  \begin{array}{c|c}
       \begin{array}{cc}
        1 & 0 \\
        y'_2 & 1 \\
       \end{array}
     & 0 \\ \hline
      \begin{array}{cc}
        Y' & 0 \\
      \end{array}
       & I_{m-2} \\
  \end{array}
\right),\]
where $X=\left(  \begin{array}{cccc}  x_3 & x_4 & \cdots & x_m \\  \end{array} \right)^t$.

If $x_1=0$ then we conclude from the braid relation
\begin{equation}\label{eqn:br121=212}
 L_1 L_2 L_1 = L_2 L_1 L_2
\end{equation}
that $y'_2=0$. But then $L_1$ and $L_2$ commute.
Again from the relation~(\ref{eqn:br121=212}) we get $L_1=L_2$, contradicting to $E^1\neq E^2$.
Hence, $x_1$ is nonzero.

Let $w_1=x_1 v_1 + x_3 v_3 + x_4 v_4 + \cdots + x_mv_m$. Let $\beta'_0$ be the basis obtained from
$\beta_0$ by replacing $v_1$ with $w_1$. With respect to this new basis we have
\[
L_1=
\left(
  \begin{array}{cc}
      U & 0 \\
      0 & I_{m-2} \\
  \end{array}
\right)=\widetilde A_1, \ \
L_2=
\left(
  \begin{array}{c|c}
       \begin{array}{cc}
        1 & 0 \\
        y_2 & 1 \\
       \end{array}
     & 0 \\ \hline
      \begin{array}{cc}
        Y & 0
      \end{array}
       & I_{m-2} \\
  \end{array}
\right),\]
where $Y=\left(  \begin{array}{cccc}  y_3 & y_4 & \cdots & y_m \\  \end{array} \right)^t$.

From the braid relation~(\ref{eqn:br121=212}), it is easy to conclude that $y_2=-1$. Now let
$w_2= v_2 - (y_3v_3 +y_4v_4+\cdots +y_m v_m)$ and let
$\beta_1$ be the basis $\beta'_0$ where $v_2$ is replaced with $w_2$. With respect to the basis
$\beta_1=\{w_1,w_2,v_3,v_4,\ldots, v_{n-1},v_m\}$ we now have
$L_1=\widetilde A_1$ and $L_2=\widetilde B_1$.

Suppose that $k<g$ and that there is a basis
\[\beta_k=\{v_1,v_2,\ldots, v_{m-1},v_m\}\]
with respect to which
\begin{equation}\label{eqn:L_i=B_i}
L_{2i-1}=\widetilde A_i \mbox{ and } L_{2i}=\widetilde B_i
\end{equation}
for all $i=1,2,\ldots,k$. Note that in this case
\[\alpha=\{v_{2k+1},v_{2k+2},\ldots, v_{m-1},v_m\}\]
is contained in $W_k=\displaystyle
\bigcap_{i=1}^{2k} E^i$. It can be shown easily that, in fact, $\alpha$ is a basis
for $W_k$, so that $\dim (W_k)=m-2k$.

Next, we consider $L_{2k+1}$ and $L_{2k+2}$. Let $s\in \{2k+1,2k+2\}$. Since the
subspace $W_k$ is $L_s$--invariant, with respect to the basis
$\beta_k$,
\[
L_s=
\left(
  \begin{array}{cc}
      Z_s & 0\\
      Y_s & X_s \\
  \end{array}
\right).\]
Here, $Z_s$ is a $2k\times 2k$ matrix. Since $L_s$ commutes with each
\[
L_{2i-1}=
\left(
  \begin{array}{cc}
      \bar A_i & 0\\
       0 & I \\
  \end{array}
\right) \ \mbox{ and } \
L_{2i}=
\left(
  \begin{array}{cc}
      \bar B_i & 0\\
       0 & I \\
  \end{array}
\right)\]
for $i=1,2,\ldots, k$,
where the matrix $\bar A_i$ is the $2k\times 2k$ block diagonal matrix
${\rm Diag} (I_2,\ldots, U,\ldots,I_2)$ whose $i^{\rm th}$ block is $U$,
and $\bar B_i$ is obtained from $\bar A_i$ by replacing $U$ with $\widehat U$,
we get that
\begin{itemize}
  \item $Z_s\bar A_i =\bar A_i Z_s$, $Y_s\bar A_i =Y_s$,
  \item $Z_s\bar B_i =\bar B_i Z_s$, $Y_s\bar B_i =Y_s$
\end{itemize}
for each $i$.
We conclude from Lemma~\ref{lem:nXYZ} that $Z_s=I_{2k}$ and $Y_s=0$, so that
\[
L_s=
\left(
  \begin{array}{cc}
      I_{2k} & 0\\
      0 & X_s \\
  \end{array}
\right)\]
In particular, $v_1,v_2,\ldots,v_{2k}$ are eigenvectors of $L_s$.

If $W_k$ were a subspace of $E^s$, then we would have $\dim (E^s)=m$. By this contradiction,
both of the subspaces $W_k\cap E^{2k+1}$ and $W_k\cap E^{2k+2}$ are of dimension $m-2k-1$.
If, furthermore, we had $W_k\cap E^{2k+1}=W_k\cap E^{2k+2}$, then we would conclude that $E^{2k+1}=E^{2k+2}$,
again arriving at a contradiction.
Hence, the subspaces $W_k\cap E^{2k+1}$ and $W_k\cap E^{2k+2}$ are different, so that
\[W_{k+1}=W_k\cap E^{2k+1}\cap E^{2k+2}\] is of dimension $m-2k-2$.

Let $\{ w_{2k+3},w_{2k+4},\ldots, w_m\}$ be a basis of $W_{k+1}$. We choose two vectors
$w_{2k+1}$ and $w_{2k+2}$ such that
\begin{itemize}
  \item $w_{2k+1}\in W_k\cap E^{2k+1}$, $w_{2k+1}\notin W_{2k+1}$,
  \item $w_{2k+2}\in W_k\cap E^{2k+2}$, $w_{2k+2}\notin W_{2k+1}$.
\end{itemize}
Then $\{ w_{2k+1},w_{2k+2},w_{2k+3},w_{2k+4},\ldots, w_{m-1}, w_m\}$
is a basis of $W_k$.
Now consider the basis
\[
\bar \beta_k=\{ v_1,v_2, \ldots, v_{2k}, w_{2k+1},w_{2k+2},\ldots, w_{m-1}, w_m\}
\]
of $\C^m$. With respect to this basis,
\begin{itemize}
  \item $L_{2i-1}=\widetilde A_i$, $L_{2i}=\widetilde B_i$ for $i=1,2,\ldots, k$,
  \item $ L_{2k+1}=
\left(
  \begin{array}{ccc}
      I_{2k} & 0 &0\\
       0 & X_1 & 0 \\
       0 & X_2 & I \\
  \end{array}
\right)$ and $ L_{2k+2}=
\left(
  \begin{array}{ccc}
      I_{2k} & 0 &0\\
       0 & Y_1 & 0 \\
       0 & Y_2 & I \\
  \end{array}
\right)$,
\end{itemize}
where
$X_1=\left(
       \begin{array}{cc}
         1 & x \\
         0 & 1 \\
       \end{array}
     \right)$,
$X_2=\left(   \begin{array}{cccc}
                 0 & 0 &  \cdots & 0 \\
                 x_{2k+3} & x_{2k+4} & \cdots & x_m\\
           \end{array}  \right)^t$,
$Y_1=\left(
       \begin{array}{cc}
         1 & 0 \\
         y'  & 1 \\
       \end{array}
  \right)$, and
  $Y_2=\left(   \begin{array}{cccc}
         y'_{2k+3} & y'_{2k+4} & \cdots & y'_m\\
         0 & 0 &  \cdots & 0 \\
    \end{array}  \right)^t$.

The rest of the proof proceeds as above: If $x=0$ then the braid relation
\begin{equation}\label{eqn:braid2k}
    L_{2k+1}L_{2k+2}L_{2k+1}=L_{2k+2}L_{2k+1}L_{2k+2}
\end{equation}
would imply that $y'=0$, so that $L_{2k+1}$ and $L_{2k+2}$ commute. The braid relation~\eqref{eqn:braid2k}
now gives $L_{2k+1}=L_{2k+2}$, which is a contradiction.

Hence, $x\neq 0$.
Define
\[
w'_{2k+1}=x \, w_{2k+1}+ (x_{2k+3} w_{2k+3}+x_{2k+4} w_{2k+4}+\cdots +x_m w_m)
\] and let
$\beta'_k$ be the basis obtained from $\bar \beta_k$ by replacing $w_{2k+1}$ with $w'_{2k+1}$.
With respect to $\beta'_k$, the matrices of $L_1,\ldots,L_{2k}$ are the same, and
$L_{2k+1}=\widetilde A_{k+1}$. The matrix of $L_{2k+2}$ turns into a new matrix of the form above
where $y'$ is replaced by some $y$, and $y'_j$ is replaced by some $y_j$. The braid relation~\eqref{eqn:braid2k} then
implies that $y=-1$. If we now define
\[
w'_{2k+2}= w_{2k+2} - (y_{2k+3}w_{2k+3} +y_{2k+4}w_{2k+4}+\cdots +y_m w_m)
\] and
\[
\beta_{k+1}=\{v_1,v_2,\ldots,v_{2k},w'_{2k+1},w'_{2k+2}, w_{2k+3},w_{2k-4},\ldots, w_{m-1},w_m\},
\] we have
$L_{2i-1}=\widetilde A_i$ and $L_{2i}=\widetilde B_i$ for all $1\leq i\leq k+1$ with respect to $\beta_{k+1}$.

Consequently, repeating this for $k=1,2,3,\ldots,g-1$, with respect to some basis of $\C^m$,
$L_{2i-1}=\widetilde A_i$ and $L_{2i}=\widetilde B_i$ for all $1\leq i\leq g$.

This finishes the proof of the lemma.
\end{proof}

\bigskip


\subsection{Triviality of a representation of mapping class group}
We give some criteria for the triviality of a representation of the mapping class group
into $\GL(m,\C)$. The main tool for this is Theorem~\ref{thm:2g-1}. This subsection
is inspired by~\cite{fh}.

\begin{lemma} \label{lem:flag}
Let $S$ be a compact connected oriented surface of genus $h\geq 2$ and let $\psi:\mod (S)\to \GL (m,\C)$ be a homomorphism.
If there is a flag $0=W_0 \subset W_1 \subset W_2 \subset \cdots \subset W_k= \C^{m}$ of $\mod (S)$--invariant subspaces
such that $\dim (W_i/W_{i-1})\leq 2h-1$ for each $i=1,2,\ldots,k$, then the image of $\psi$ is
\begin{itemize}
  \item[(i)] trivial if $h\geq 3$, and
  \item[(ii)] a quotient of $\Z_{10}$ if $h=2$.
\end{itemize}
Equivalently, the image of the commutator subgroup of $\mod (S)$ is trivial.
\end{lemma}
\begin{proof} We set $\Gamma=\mod (S)$.
For each $i=1,2,\ldots, k$, let $\alpha_i$ be a basis of $W_i$ such that
$\alpha_i \subset \alpha_{i+1}$. We work with the basis $\alpha_k$ of $W_k=\C^m$.
Since each $W_i$ is $\Gamma$--invariant, for $f\in \Gamma$ the matrix $\phi (f)$ is of the form
\[ \psi (f)=
\left(
  \begin{array}{ccccc}
    F_1 & \star & \star &\cdots & \star \\
    0 & F_2 & \star & \cdots & \star \\
    0 & 0 & F_3 & \cdots & \star \\
    \vdots & \vdots &\vdots & \ddots & \vdots \\
    0 & 0 &0& \cdots & F_k \\
  \end{array}
\right).\]
Then, for each $i$, the correspondence $f\mapsto F_i$ defines a homomorphism $\psi_i: \Gamma \to \GL (W_i/W_{i-1})$. Since
$\dim (W_i/W_{i-1})\leq 2h-1$, the image of $\psi_i$ is cyclic (trivial if $h\geq 3$)
by Theorem~\ref{thm:2g-1}.

Since the image $\psi_i$ is abelian, we get $\psi_i(f)=I$ for all $f\in [\Gamma,\Gamma]$ and for all $i$. Hence,
for any $f\in [\Gamma,\Gamma]$, $\psi (f)$ is upper triangular with $1$ along the diagonal. The subgroup
of $\GL(m,\C)$ consisting of such matrices is nilpotent, and the group $[\Gamma,\Gamma]$ is perfect by Theorem~\ref{thm:G'prf}.
We conclude from this that $\psi( [\Gamma,\Gamma])$ is trivial.
\end{proof}

\begin{corollary} \label{cor:flag}
 Let $R$ be a compact connected oriented
surface of genus $g-1 \geq 3$ and let $\psi:\mod (R)\to \GL (2g,\C)$ be a homomorphism.
If there is a $\mod (R)$--invariant subspace $W$ with $3\leq \dim (W)\leq 2g-3$,
then $\psi $ is trivial.
\end{corollary}

\bigskip


\section{Uniqueness of the symplectic representation}
\label{sec:symprep}
Let $S$ be a compact oriented surface of genus $g\geq 3$ and let $\phi:\mod (S)\to \GL (2g,\C)$ be
a homomorphism. We prove in this section that either $\phi$ is trivial or, with respect to a suitable basis of $\C^{2g}$,
the image of $\phi$ is equal to $\Sp(2g,\Z)$ with respect to a suitable basis of $\C^{2g}$. Theorem~\ref{thm:main1}
will follow from this.
Recall that for a simple closed curve $a$, $E^{a}_\lambda$ denotes the eigenspace
corresponding to an eigenvalue $\lambda$ of $L_a=\phi (t_a)$ and $\lambda_\#$ denotes the multiplicity
of $\lambda$.

\begin{lemma}
\label{lem:E=2g-2}
Let $g\geq 4$, let $a$ be a nonseparating simple closed curve on $S$ and let $\lambda$ be an
eigenvalue of $L_a$. If $\lambda_\#\geq 3$ then $\dim (E^{a}_\lambda) \geq 2g-1$.
In particular, $\lambda_\#\geq 2g-1$.
\end{lemma}
\begin{proof}
Let $b$ be a nonseparating simple closed curve on $S$ intersecting $a$ at one point, and
let $R$ denote the complement of a regular neighborhood of $a\cup b$, so that
it is a subsurface of genus $g-1$. Then, $\mod(R)$ injects into $\mod(S)$.
By identifying $\mod(R)$ with its image, we assume that $\mod (R)$ is a subgroup of $\mod (S)$.

Suppose first that $\dim (E^{a}_\lambda) \leq 2g-3$. Define a subspace $W$ by
\[
W=\left\{
\begin{array}{ll}
  \ker(L_a-\lambda I)^{\lambda_\#}, & \mbox{if } \lambda_\#\leq 2g-3,\\
  \ker(L_a-\lambda I)^3, & \mbox{if } \lambda_\# \geq 2g-2 \mbox{ and } \dim(E^a_\lambda)=1, \\
  \ker(L_a-\lambda I)^2, & \mbox{if } \lambda_\# \geq 2g-2 \mbox{ and } \dim(E^a_\lambda)=2, \\
  E^a_\lambda,           & \mbox{if } \lambda_\# \geq 2g-2 \mbox{ and } 3\leq \dim(E^a_\lambda)\leq 2g-3.
\end{array}
\right.
\]
Since elements of $\mod (R)$ commute with the Dehn twist $t_a$, the subspace $W$ is $\mod(R)$--invariant
and its dimension satisfies $3\leq \dim (W) \leq 2(g-1)-1$.
Since the genus of $R$ is $g-1\geq 3$, $\phi (\mod(R))$ is trivial by Corollary~\ref{cor:flag}.
Since $t_a$ is conjugate to some Dehn twist in $\mod (R)$, we get that $L_a=I$.
This says, in particular, that $\dim (E^a_\lambda) = 2g$, which is a contradiction.

Suppose now that $\dim (E^{a}_\lambda) = 2g-2$.
If $E_\lambda^a \neq E_\lambda^b$, then $E_\lambda^a \cap E_\lambda^b$ is a $\mod (R)$--invariant subspace
and its dimension is either $2g-3$ or $2g-4$.
Hence, $\phi$ is trivial on $\mod (R)$ by Corollary~\ref{cor:flag}.
We conclude again that $L_a=I$, obtaining a contradiction.
If $E_\lambda^a = E_\lambda^b$ then by Lemma~\ref{lem:E-a=E-b}
the eigenspace $E_\lambda^a $ is a $\mod(S)$--invariant subspace of dimension $2g-2$. Hence,
$0\subset E_\lambda^a \subset \C^{2g}$ is a $\mod(S)$--invariant flag.
Now Lemma~\ref{lem:flag} applies to conclude that $\phi$ is trivial, arriving at a contradiction again.
\end{proof}

\begin{lemma}
\label{lem:es=2g-1}
Let $g\geq 4$, let $a$ be a nonseparating simple closed curve on $S$ and let $\lambda$ be an
eigenvalue of $L_a$. If $\dim (E^{a}_\lambda) = 2g-1$ then $\lambda=1$.
\end{lemma}
\begin{proof}
Let $a=c_1$. Choose six nonseparating simple closed curves $c_2,c_3,c_4,c_5,c_6,c_7$ on $S$ such that
we have a lantern relation
\begin{equation}
    t_{c_1}t_{c_2}t_{c_3}t_{c_4}=t_{c_5}t_{c_6}t_{c_7}.
\end{equation}
Hence,
\begin{equation}\label{eqn:wwx}
    L_{c_1}L_{c_2}L_{c_3}L_{c_4}=L_{c_5}L_{c_6}L_{c_7}.
\end{equation}
The subspace
$\bigcap_{i=1}^7 E^{c_i}_\lambda$ has positive dimension. Let $v$ be a nonzero vector in this intersection.
Evaluating both sides of~\eqref{eqn:wwx} to $v$ gives $\lambda^4 v=\lambda^3v$, concluding $\lambda=1$.
\end{proof}

\subsection{Proof of Theorem~\ref{thm:main1} for $g\geq 4$}
We give the proof in several steps.
Let $a$ be a nonseparating simple closed curve on $S$
and let $L_a=\phi(t_a)$.

\textbf{Step 1:} We claim that $L_a$ has only one eigenvalue.
First of all, $L_a$ has at most two eigenvalues by Corollary~\ref{cor:2ev}.
Suppose that it has two eigenvalues $\lambda$ and $\mu$ with $\lambda_\#\geq \mu_\#$. Since
$\lambda_\# + \mu_\#=2g$, we have $\mu_\#\leq g$ and $\lambda_\#\geq g$.
Then, we get from Lemma~\ref{lem:ev=1} that $\mu=1$, so that $\lambda\neq 1$.
 On the other hand, by Lemma~\ref{lem:E=2g-2},
$\lambda_\# = 2g-1$ and the dimension of the eigenspace $E^a_\lambda$ is $2g-1$.
Now Lemma~\ref{lem:es=2g-1} implies that $\lambda=1$, giving the desired contradiction.
Therefore, $L_a$ has only one eigenvalue, say $\lambda$.

 By Lemma~\ref{lem:E=2g-2},
the dimension of the eigenspace $E^a_\lambda$ is either $2g-1$ or $2g$.

\textbf{Step 2:}
Suppose first that $\dim(E^a_\lambda)= 2g$, i.e, $L_a=\lambda I$. For any nonseparating simple closed curve $x$,
the Dehn twist $t_x$ is conjugate to $t_a$, so that $L_x$ is conjugate to $L_a$, implying that $L_x=\lambda I$.
Since the mapping class group $\mod (S)$ is generated by Dehn twists about nonseparating simple closed curves,
it follows that the image of $\phi$ is cyclic. Since the mapping class group
$\mod (S)$ is perfect, the image of $\phi$ is trivial.

\textbf{Step 3:} Suppose finally that $\dim(E^a_\lambda) = 2g-1$. Lemma~\ref{lem:es=2g-1} implies that $\lambda=1$.
Thus, the Jordan form of $L_a$ is as in~\eqref{eqn:Jordan}.
Now by Lemma~\ref{lem:L=AL=B}, with respect to a basis of $\C^{2g}$, we have
$L_{a_i}=A_i$ and $L_{b_i}=B_i$ for each $i=1,2,\ldots, g$.

\textbf{Step 4:} Let $\bar S$ be the surface obtained from $S$ by gluing a disk along each boundary component
and let $\pi:\mod(S) \to \mod (\bar S)$ be the surjective homomorphism obtained by extending
diffeomorphisms by the identity over the glued disks. Consider the curves illustrated in Figure~\ref{abcurves}.
Let $e\in \{ e_0,e_1,\ldots,e_p,f_0,f_1,\ldots,f_r \}$.
Since $e$ intersects $b_1$ at one point and is disjoint from all $a_i$ and $b_j$, $j\geq 2$,
the matrix $L_e=\phi (t_{e})$ commutes with all $A_i$ and all $B_j$,
and satisfy the braid relation $ L_{e}B_1L_{e}=B_1L_{e}B_1$. Since $L_e$ is conjugate to
$L_a$, it has only one eigenvalue $\lambda=1$. From Lemma~\ref{lem:X=A-1}, we get that $L_e=A_1$.
It now follows from Proposition~\ref{prop:kernel} that the kernel of
$\pi$ is contained in the kernel of $\phi$. Hence, $\phi$ induces a homomorphism
$\bar \phi: \mod (\bar S) \to \GL (2g,\C)$.

\textbf{Step 5:} Note also that $\phi (t_{a'_{g-1}})=A_{g-1}=\phi (t_{a_{g-1}})$, giving
$\pi (t_{a_{g-1}}^{-1} t_{a'_{g-1}})=I$. A celebrated result of Johnson~\cite{dlj}
says that the Torelli subgroup of $\mod (\bar S)$ is generated normally by
$\pi (t_{a_{g-1}}^{-1} t_{a'_{g-1}})$. Therefore, the Torelli subgroup of $\mod(\bar S)$
is contained in the kernel of $\bar \phi$, so that $\bar \phi$ induces a map
$\varphi :\Sp (2g,\Z)\to \GL (2g,\C)$.

$$
\begindc{\commdiag}[7]
\obj(10,11)[A]{$\mod (S)$}
\obj(10,6) [B]{$\mod (\bar S)$}
\obj(10,1)[C]{$\Sp(2g,\Z)$}
\obj(25,6)[D]{$\GL(2g,\C)$}
\mor{A}{B}{$\pi$}[1,7]
\mor{B}{C}{$ $}[1,7]
\mor{A}{D}{$\phi$}[1,0]
\mor{B}{D}{$\bar\phi$}[1,0]
\mor{C}{D}{$\varphi$}[1,0]
\enddc
$$

\textbf{Step 6:} As the matrix $L_a$ has infinite order, $\Im (\varphi)=\Im (\phi )$ is infinite.
One may easily check that $(A_1B_1A_2B_2\cdots A_gB_g)^3=-I$ so that
the kernel of $\varphi$ does not contain $-I$. From the solution of the congruence subgroup
problem for $\Sp (2g,\Z)$ (see~\cite{menicke}, Corollar 1.), we conclude that the kernel of $\varphi$ is trivial,
so that $\varphi$ is injective.

This concludes the proof of Theorem~\ref{thm:main1}.

\subsection{Sketch of the proof of Theorem~\ref{thm:main1} for $g= 3$}
The proof of Theorem~\ref{thm:main1} for $g=3$ requires more detailed analysis of eigenvalues.
We only sketch the proof.
Note that in this case $\phi:\mod(S)\to \GL(6,\C)$.

Let $a$ and $b$ be two nonseparating simple closed curves intersecting at one point, let $R$
be the complement of a regular neighborhood of $a\cup b$,
and let $c_1$ and $c_2$ be two simple closed curves on $R$ intersecting each other at one point.

\begin{figure}[hbt]
 \begin{center}
      \includegraphics[width=6cm]{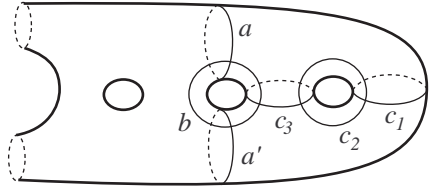}
      \caption{Curves on the surface of genus $3$.}
      \label{g=3curves}
   \end{center}
 \end{figure}

In this case, $L_a$ has at most two eigenvalues by Corollary~\ref{cor:2ev} again.
Suppose first that $L_a$ has two eigenvalues $\lambda$ and
$\mu$ with $\lambda_\#\geq \mu_\#$. Then $\lambda_\#\geq 4$, $\mu_\#\leq 2$, $\mu=1$ and
$\dim (E^a_\mu)=\mu_\#$.
Let $r=\dim(E^a_\lambda)$. Consider the following $\mod(R)$--invariant flags:
\begin{itemize}
  \item $0\subset \ker (L_a-\lambda I)^3\subset \C^6$ if $r=1$,
  \item $0\subset \ker (L_a-\lambda I)\subset \ker (L_a-\lambda I)^2 \subset \C^6$ if $r=2$,
  \item $0\subset E^a_\lambda \subset \C^6$ if $r=3$.
\end{itemize}
Therefore, in the cases $r\leq 3$, Lemma~\ref{lem:flag}  gives that the commutator subgroup of $\mod(R)$ is in the kernel of $\phi$.
In particular, $\phi (t_{c_1}t_{c_2}^{-1})=I$. Since the normal closure of $t_{c_1}t_{c_2}^{-1}$ in $\mod(S)$ is $\mod(S)$ (c.f.~Theorem~\ref{thm:G'}),
the map $\phi$ is trivial, so that $L_a=I$, a contradiction.
In the case  $r=4$, the flag
 $0\subset E^a_\lambda \cap E^b_\lambda \subset E^a_\lambda\subset \C^6$ is $\mod(R)$--invariant
 if $E^a_\lambda\neq E^b_\lambda$, and  the flag
 $0\subset E^a_\lambda \subset \C^6$ is $\mod(S)$--invariant if $E^a_\lambda = E^b_\lambda$,
 both implying that $\phi$ is trivial, a contradiction again.

Suppose now that $r=5$. Fix a basis of $\C^6$ such that
$L_a=\left(
       \begin{array}{cc}
         \lambda I_5 & 0 \\
         0 & 1 \\
       \end{array}
     \right)$.
Since the first homology group of $\mod(S)$ is trivial, $\lambda^5=1$.
For any simple closed curve $x$ disjoint from $a$, the matrix of $L_x$ is equal to
$L_x=\left(
       \begin{array}{cc}
         * & 0 \\
         0 & \gamma(x) \\
       \end{array}
     \right)$
for some nonzero complex number $\gamma (x)$. If $x$ is a nonseparating curve, then $\gamma (x)$ is either $1$ or $\lambda$, as it is an eigenvalue
of $L_x$.

Consider the curves illustrated in Figure~\ref{g=3curves}. If $\gamma (c_1)=1$ then $E^{c_1}_1=E^a_1$.
By the braid relation, it can be seen that $\gamma (c_2)=1$. Hence, $E^{c_2}_1=E^a_1$. It follows that
$E^{x}_1=E^a_1$ for all nonseparating $x$ on $S$, so that $E^a_1$ is $\mod(S)$--invariant.
One may deduce from this that $\phi$ is trivial. In particular, $\lambda=1$.
If $\gamma (c_1)=\lambda$ then braid relations between $t_{c_i}$ imply that $\gamma (c_2)=\gamma (c_3)=\lambda$. Now the relation
$(t_{c_1}t_{c_2}t_{c_3})^4=t_{a}t_{a'}$ and $\lambda^5=1$ imply that $\gamma (a')=\lambda^2$, which is an eigenvalue of $L_{a'}$.
On the other hand, since all eigenvalues of $L_{a'}$ are $1$ and $\lambda$, we conclude that $\lambda=1$.
By these contradictions, $L_a$ must have only one eigenvalue, say $\lambda$.

If $\dim(E^a_\lambda)\leq 4$ then one arrives at a contradiction by deducing that $\phi$ must be trivial.
If $\dim(E^a_\lambda)=6$ then $\phi$ is trivial.
If $\dim(E^a_\lambda)=5$ then Lemma~\ref{lem:ev=1for2g-1} gives $\lambda=1$. Now the rest of the proof proceeds
as in the case $g\geq 4$.
\qed

\bigskip


\section{Non--faithful representations of $\mod(S)$: Proof of Theorem~\ref{thm:main2}}
\label{sec:nonfaithful}

Our aim in this section is to prove Theorem~\ref{thm:main2}. We do this by proving a
stronger result, Theorem~\ref{thm:2g+n}, which says that certain terms in the derived series,
defined below,
of the Torelli subgroup of the mapping class group of a subsurface of genus three are
contained in the kernel of $\phi$. Since the Torelli subgroups of mapping class
groups we consider are not solvable (they contain
nonabelian free groups),  Theorem~\ref{thm:2g+n} will certainly imply
Theorem~\ref{thm:main2}. Recall that the \emph{Torelli subgroup}
of the mapping class group of an orientable surface with at most one boundary component is the subgroup
consisting of those mapping classes which act trivially on the first homology of
the surface.

For a subsurface $X$ of $S$ with one boundary component, the inclusion $X \hookrightarrow S$
induces a homomorphism $\mod (X)\rightarrow \mod (S)$. This map is injective when the genus
of $X$ is less than that of $S$. When this is the case,
we identify the group $\mod (X)$ with its image in $\mod (S)$ as usual. We denote by $T_X$ (the image of) the
Torelli subgroup of $\mod (X)$.

For a group $G$, let $G^{(0)}=G$. For each integer $k\geq 1$, we inductively define the
$k^{\rm th}$ derived subgroup $G^{(k)}$ of $G$ by
\[G^{(k)}=[G^{(k-1)},G^{(k-1)}].\]
Recall that a group $G$ is called \emph{solvable} if $G^{(k)}=1$ for some $k$.

\begin{theorem} \label{thm:2g+n}
Let $n \geq 0$ be an integer. If $g\geq n+3$ and if \ $\phi : \mod (S)\to \GL (2g+n,\C)$ is a homomorphism,
then the $n^{\rm th}$ derived subgroup $T_X^{(n)}$ of $T_X$
is contained in the kernel of $\phi$ for any genus--$3$ subsurface $X$ of $S$
with one boundary component.
\end{theorem}
\begin{proof}
We prove the theorem by induction on $n$. If $n=0$ then $g\geq 3$, so the result follows from Theorem~\ref{thm:main1}.

Let $n\geq 1$, $g\geq n+3$ and let $\phi:\mod (S)\to \GL (2g+n,\C)$ be a homomorphism.
Suppose that the theorem holds true for all $k\leq n-1$.
If $X$ and $Y$ are any genus--$3$ subsurfaces $S$ each
with one boundary component, then the classification of surfaces
tells us that there is a self--diffeomorphism of $S$ mapping $X$ to $Y$. It follows that
the subgroups $\mod(X)$ and $\mod (Y)$ of $\mod (S)$ are conjugate. It is now easy to conclude that
it suffices to prove the theorem for some genus--$3$ subsurface.

Let $a$ be a nonseparating simple closed curve on $S$. Since $2g+n\leq 4g-5$,  by Corollary~\ref{cor:2ev} the matrix
 $L_a$ has at most two eigenvalues.

Let $b$ be a nonseparating simple closed curve intersecting $a$ at one point and let
$R$ denote the complement of a regular neighborhood of $a\cup b$, so that
$R$ is a compact surface of genus $g-1$. By extending self-diffeomorphisms
$R$ to $S$ by the identity, we regard $\mod (R)$ as a subgroup of $\mod (S)$.

 \emph{Case 1: $L_a$ has two eigenvalues.} Suppose that $L_a$ has two eigenvalues $\lambda$ and $\mu$,
 with $\lambda_\#  \geq\mu_\#$. Recall that $\lambda_\# $ denotes the multiplicity of the eigenvalue $\lambda$.
 Since $\lambda_\#  + \mu_\# =2g+n \leq 3g-3$, we have
 $\mu_\# < 2g-3$. By Lemma~\ref{lem:ev=1}, we have $\mu=1$
 and $\dim (E^a_\mu)=\mu_\#$. The dimension of $\ker (L_a -\lambda I)^{\lambda_\#}$ is
 ${\lambda_\#}$. Let $\alpha_1$ be a (ordered) basis of $\ker (L_a -\lambda I)^{\lambda_\#}$
 and let $\alpha_2$ be a basis of $E^a_1$. Then $\alpha=\alpha_1\cup \alpha_2$
 is a basis of $\C^{2g+n}$ and in this basis the matrix of $L_a$ is of the form
 \begin{equation} \label{eqn:matrix1}
  \left(
  \begin{array}{cc}
    \star & 0 \\
    0  & I \\
  \end{array}
\right),
\end{equation}
where the size of the identity matrix $I$ is $\mu_\#$. We work with this basis.

For any $f\in \mod (R)$, $\phi (f)$ commutes with
$L_a$ so that the subspaces $\ker (L_a-\lambda I)^{\lambda_\#}$ and $E^a_1=\ker (L_a-I)$
are $\phi (f)$--invariant, so
\[ \phi (f)
 = \left(
  \begin{array}{cc}
    * & 0 \\
    0  & F' \\
  \end{array}
\right),\]
where $F'$ is of size $\mu_\#$. This way we get a homomorphism $\varphi:\mod (R) \to \GL (\mu_\#,\C)$ defined by
$\varphi(f)=F'$.
Note that since $n\geq 1$, the genus of $S$ satisfies $g\geq 4$, so that the genus of $R$ is $g-1\geq 3$.
By applying Theorem~\ref{thm:2g-1} to $\varphi$, we conclude that $F'=I$. It then follows
that for any simple closed curve $x$ that is nonseparating
on $R$, the eigenspace of $L_x$ satisfies $E^x_1=E^a_1$. In particular, $L_x$ is of the from~\eqref{eqn:matrix1}.

If $c$ and $d$ are two nonseparating simple closed curves on $R$ intersecting at one point,
then we have $E^c_1=E^d_1$ $(=E^a_1)$. Now apply Lemma~\ref{lem:E-a=E-b} to conclude that
$E^a_1$ is $\mod (S)$--invariant. Since the dimension of $E^a_1$ is less than $2g$,
the action of $\mod (S)$ on $E^a_1$ is trivial. In particular, for any element $f\in\mod (S)$, the matrix of $\phi(f)$ is
\[ \phi (f) = \left(
  \begin{array}{cc}
    F & 0 \\
    *  & I \\
  \end{array}
\right).\]
It follows that $\phi$ gives rise to a homomorphism
 $\bar \phi :\mod (S)\to \GL (\lambda_\#,\C)$ defined by
\[ \phi (f) = \left(
  \begin{array}{cc}
    \bar \phi (f) & 0 \\
    *  & I \\
  \end{array}
\right).\]
Since $\lambda_\# \leq 2g+n-1$, by the induction hypothesis $T_X^{(n-1)}$ is contained in $\ker (\bar \phi)$
for some genus--$3$ subsurface $X$ of $S$. We conclude from this
that $\phi (T_X^{(n-1)})$ is abelian, and hence, $\phi (T_X^{(n)})$
is trivial.

\emph{Case 2: $L_a$ has only one eigenvalue.}
 Suppose that $L_a$ has only one eigenvalue, say $\lambda$. Let $X$ be any genus--$3$ subsurface of $R$
with one boundary component. We may consider $\mod(X)$ as a subgroup of $\mod(R)$.

We claim that if there is a $\mod (R)$--invariant subspace $V$ of dimension $r_1$ with $3\leq r_1 \leq 2g+n-3$,
then we are done, namely $\phi (T_X^{(n)})$ is trivial. For the proof of the claim, suppose that
we have such a subspace $V$. Let $r_2=2g+n-r_1$ and
let $\alpha$ be a (ordered) basis of $\C^{2g+n}$ such that first $r_1$ elements span $V$.
With respect to this basis, for any $f\in \mod (R)$,
\[ \phi (f)= \left(
  \begin{array}{cc}
    \phi_1 (f) & \star \\
    0  & \phi_2 (f) \\
  \end{array}
\right),\]
so that we have two homomorphisms $\phi_1:\mod (R) \to \GL (V)=\GL (r_1,\C)$ and
$\phi_2:\mod (R) \to \GL (\C^{2g+n}/V)=\GL (r_2,\C)$.
Since both $r_1$ and $r_2$ satisfy $r_1\leq 2(g-1)+n-1$ and $r_2\leq 2(g-1)+n-1$ and since $g-1\geq (n-1)+3$,
the derived subgroup $T_X^{(n-1)}$ of $T_X$ is contained in the kernel of both $\phi_i$ by the induction hypothesis.
That is to say,
\[ \phi(f)= \left(
  \begin{array}{cc}
    I_{r_1} & * \\
    0  & I_{r_2} \\
  \end{array}
\right)\]
for any $f\in T_X^{(n-1)}$. It follows that $\phi (T_X^{(n-1)})$ is abelian, and hence
$\phi (T_X^{(n)})$ is trivial, proving the claim.

Let us set $r=\dim(E^a_\lambda)$ for the rest of the proof, so that $1\leq r\leq 2g+n$.

Suppose first that $r\leq 2g+n-3$.
If $r$ is equal to $1$ or $2$, then $3\leq \dim (\ker(L_a-\lambda I)^3) \leq 6$.
It follows that in this case, there is a $\mod (R)$--invariant
subspace $V$ with $3\leq \dim(V)\leq 2g+n-3$. Hence, by the above claim, $\phi (T_X^{(n)})$ is trivial.

Suppose next that $r=2g+n$, so that $L_a=\lambda I$. Since $t_x$ is conjugate to $t_a$
for any nonseparating simple closed curve $x$ on $S$, $L_x=\lambda I$. Since $\mod (S)$ is generated by such Dehn
twists, it follows that the image of $\phi$ is cyclic. Since the first homology group of $\mod (S)$ is trivial,
we conclude that the image of $\phi$ is trivial.

Suppose now that $r=2g+n-2$. The dimension of $E^b_\lambda$ is also $r$.
If $E^a_\lambda \neq E^b_\lambda$ then $E^a_\lambda \cap E^b_\lambda$ is a $\mod (R)$--invariant subspace
of dimension $2g+n-3$ or $2g+n-4$. Hence, $T_X^{(n)}$ is contained in the kernel of
$\phi$ by the claim above.
If $E^a_\lambda=E^b_\lambda$ then it follows from Lemma~\ref{lem:E-a=E-b} that
$E^a_\lambda$ is $\mod (S)$--invariant.
Let $\beta$ be an ordered basis of $\C^{2g+n}$ whose first $r$ elements
form a basis of $E^a_\lambda$. For any $f\in\mod(S)$, the matrix of $\phi (f)$
with respect to $\beta$ is
\[ \phi(f)= \left(
  \begin{array}{cc}
    \phi_1 (f) & \star \\
    0  & \phi_2 (f) \\
  \end{array}
\right).\]
In this way we get two homomorphisms $\phi_1: \mod (S)\to\GL (r,\C)$ and $\phi_2: \mod (S)\to\GL (2,\C)$.
By Theorem~\ref{thm:2g-1}, the image of $\phi_2 $ is trivial, and by the induction hypothesis,
$T_X ^{(n-2)}$ is contained in the kernel of $\phi_1$. Therefore, for any $f\in T_X ^{(n-2)}$,
\[ \phi(f)= \left(
  \begin{array}{cc}
    I_r & \star \\
    0  & I_2\\
  \end{array}
\right).\]
It follows that $\phi (T_X ^{(n-2)})$ is abelian, and $\phi (T_X ^{(n-1)})$ is trivial.
In particular,  $\phi (T_X ^{(n)})$ is trivial.

Suppose finally that $r=2g+n-1$. By Lemma~\ref{lem:ev=1for2g-1}, we get $\lambda=1$.
If $E^a_1=E^b_1$ then by Lemma~\ref{lem:eigsp=} we get that  $E^x_1=E^a_1$ for all nonseparating
simple closed curves $x$ on $S$. Since the group $\mod(S)$ is generated by Dehn twists about nonseparating
simple closed curves, it follows that every element of $\mod (S)$ acts trivially on $E^a_1$.
With respect to a basis of $\C^{2g+n}$ whose first $r=2g+n-1$ elements belong to $E^a_1$,
the matrices of $L_a$ and $L_b$ are of the form
\[ \left(
  \begin{array}{cc}
    I_r & \star \\
    0  & 1\\
  \end{array}
\right).\]
It follows that $L_aL_b=L_bL_a$. From the braid relation $L_aL_bL_a=L_bL_aL_b$, we get $L_a=L_b$,
and so $\phi (t_at_b^{-1})=I$. Since the normal closure of $t_at_b^{-1}$ in $\mod (S)$
is the whole group, we conclude that $\phi$ is trivial, which is a contradiction to $\dim (E^a_1)=2g+n-1$.

Therefore, we have $E^a_1\neq E^b_1$. Now, by Lemma~\ref{lem:L=AL=B}, we have
$L_{a_i}=\left(
 \begin{array}{cc}
    A_i & 0 \\
    0  & I_n\\
  \end{array}
\right)$ and
$L_{b_i}=\left(
 \begin{array}{cc}
    B_i & 0 \\
    0  & I_n\\
  \end{array}
\right)$ for a suitable basis of $\C^{2g+n}$. It can now be concluded as in the proof of Theorem~\ref{thm:main1} that
the kernel of $\mod(S)\to \Sp (2g,\Z)$ is contained in the kernel of $\phi$.
In particular, $T_X$, and hence  $T_X^{(n)}$, is contained in the kernel of $\phi$.

This completes the proof of the theorem.
\end{proof}

\end{document}